\documentclass{article}
\usepackage{amsmath, amssymb, mathrsfs, amsthm}
\usepackage{tikz}
\usetikzlibrary{calc}
\usepackage{enumerate}

      \theoremstyle{plain}
      \newtheorem{theorem}{Theorem}[section]
      \newtheorem*{theorem*}{Theorem}
	  
	  \newtheorem*{claim*}{Claim}
      \newtheorem{proposition}[theorem]{Proposition}
      \newtheorem*{proposition*}{Proposition}
      \newtheorem{lemma}[theorem]{Lemma}
      \newtheorem*{lemma*}{Lemma}
      \newtheorem{corollary}[theorem]{Corollary}
      \newtheorem*{corollary*}{Corollary}
      
      \theoremstyle{definition}
      \newtheorem{definition}[theorem]{Definition}
      \newtheorem*{definition*}{Definition}
      
      \theoremstyle{remark}
      \newtheorem{remark}[theorem]{Remark}
      \newtheorem*{remark*}{Remark}
      
      \newtheorem*{example*}{Example}

\newcommand{\R}{\mathbb{R}}
\newcommand{\C}{\mathbb{C}}
\newcommand{\N}{\mathbb{N}}
\newcommand{\Z}{\mathbb{Z}}

\renewcommand{\tilde}{\widetilde}
\newcommand{\abs}[1]{\left| #1 \right|}
\newcommand{\norm}[1]{\left\lVert #1 \right\rVert}

\renewcommand{\O}{\Omega}
\renewcommand{\hat}{\widehat}

\renewcommand{\sp}{\operatorname{sp}}
\newcommand{\A}{\mathscr{A}}
\newcommand{\bm}{\left(\begin{matrix}}
\newcommand{\sm}{\end{matrix}\right)}
\newcommand{\Id}{\operatorname{Id}}

\newcommand{\Tinf}{{\operatorname{T}_\infty}}
\newcommand{\Tone}{{\operatorname{T}_1}}

\numberwithin{equation}{section}

\author{Eemeli Bl{\aa}sten\thanks{Research supported by the Mittag-Leffler Institute}, Lassi P\"aiv\"arinta\thanks{Research supported by the ERC 2010 Advanced Grant 267700}}
\title{Completeness of generalized transmission eigenstates}

\begin{document}
\maketitle

\section*{Abstract}
We prove the completeness of the generalized (interior) transmission eigenstates for the acoustic and Schr\"odinger equations. The method uses the ellipticity theory of Agranovich and Vishik.

\newpage
\section{Introduction}
In this paper we prove the completeness of the generalized eigenstates corresponding to the interior transmission problem
\begin{align}
\label{helmholtzEQ}
\big(\Delta + k^2(1+m)\big) w = 0& \quad \text{in } \Omega\\
(\Delta + k^2) v = 0& \quad \text{in } \Omega\\
\label{boundaryEQ1}
u = \tfrac{\partial}{\partial \nu} u = 0& \quad \text{on } \partial\Omega
\end{align}
for $u = w-v$ with $m>0$ smooth in $\overline{\O}$. For an overview of transmission eigenvalues we refer to \cite{CCbook}, \cite{CakoniGintidesHaddar}, \cite{CPS}, \cite{HitrikKrupchykOlaPaivarintaTEconstant2010} and \cite{PS}. Cakoni, Gintides and Haddar posed the question of completeness for the eigenstates in \cite{CakoniGintidesHaddar}. A sufficient condition for completeness of the generalized eigenstates for higher order elliptic operators is given in \cite{HitrikKrupchykOlaPaivarintaEllipticOps2011}.

Our method gives also a new proof for the existence of an infinite number of transmission eigenvalues and an upper bound for their counting function. While the mathematics of the present paper were already finished, we became aware of the manuscript of Luc Robbiano \cite{robbiano}, where similar results were shown. In contrary to Robbiano's proof, our argument is based Agranovich and Vishik's ellipticity condition \cite{AgranovichVishik}. Our method generalizes also to many other interior transmission problems than \eqref{helmholtzEQ} -- \eqref{boundaryEQ1}, for example the Schr\"odinger case
\begin{align}
\big((\Delta + k^2) + V\big)w = 0 & \quad \text{in } \Omega\\
(\Delta + k^2) v = 0 & \quad \text{in } \Omega\\
\label{boundaryEQ2}
u = \tfrac{\partial}{\partial \nu} u = 0& \quad \text{on } \partial\Omega
\end{align}
with $u = w - v$ again.
What is somewhat surprising, is that the boundary conditions \eqref{boundaryEQ1} and \eqref{boundaryEQ2} can be replaced with more general boundary conditions
\begin{equation}
\begin{aligned}
\partial_\nu^{m_1} u &= 0\\
\partial_\nu^{m_2} u &= 0
\end{aligned} \quad \text{on } \partial \Omega
\end{equation}
where $m_1, m_2 \in \{0, 1, 2, 3\}$ are different. See Section \ref{lastSection}.

Let us denote $\lambda = - k^2$ and $q = \tfrac{1}{m}$. It is well-known that the interior transmission problem \eqref{helmholtzEQ} -- \eqref{boundaryEQ1} with $w,v \in H^4(\Omega)$ can be reduced to seeking for nontrivial solutions $u \in H^4(\Omega)$ of the fourth order equation
\begin{align}
\label{pencil1}
T(\lambda) u = \big(\Delta - \lambda(1 + \tfrac{1}{q})\big) q  (\Delta - \lambda) u &= 0 \quad \text{in } \Omega\\
\label{pencil2}
u &\in H^2_0(\Omega)
\end{align}
From now on we shall only consider such operator pencils instead of the eigenvalue problems. We show that the problem \eqref{pencil1} -- \eqref{pencil2} is parameter-elliptic and its inverse is meromorphic of finite order. Completeness then follows from the Phragm\'en-Lindel\"of principle. Note that this is despite the fact that the problem \eqref{helmholtzEQ} -- \eqref{boundaryEQ1} is not even elliptic. Similar methods have already been used in the study of transmission eigenvalue problems, for example when establishing the location of transmission eigenvalues in the complex plane \cite{HitrikKrupchykOlaPaivarintaTElocations2011}. Lakshtanov and Vainberg \cite{LakshtanovVainbergBounds2012}, \cite{LakshtanovVainbergEllipticity2012}, \cite{LakshtanovVainberg} have also considered parameter-ellipticity in relation to transmission eigenvalues, but without using the operator pencil \eqref{pencil1}.

We note that the completeness of the generalized eigenstates would also follow directly from parameter-ellipticity by Theorem 5 in \cite{Agranovich1976} or more explicitly by Theorem 2 in \cite{Agranovich1990}. See also Section 6 in \cite{Agranovich1992}. Hence this proof could be reduced to just proving parameter-ellipticity. Nevertheless we have decided to keep the rest of the paper for a more self-contained article and because our method may work also for analytic families of operators $T(\lambda)$ that are not polynomial pencils. Another very interesting future work is studying the paper \cite{Agranovich1992} for getting a precise Weyl law instead of the upper bound in Theorem \ref{WeylLawThm}.

\subsection*{Aknowledgements}
The authors are grateful for Katya Krupchyk and Michael Hitrik for useful discussion and suggesting to use the Agranovich and Vishik method. We would also like to thank the referees for the two observations leading to the last paragraph of the introduction.

\newpage
\section{Idea of the proof}
The original idea on how to prove this result was to combine the ellipticity result of Agranovich and Vishik \cite{AgranovichVishik} with the methods of Robert and Lai \cite{LaiRobertNL,LaiRobertENG}. Basically the idea was to estimate $\norm{T(\lambda)^{-1}}$ as $\abs{\lambda} \to \infty$ on some rays and show that $T(\lambda) = A_0 + \lambda A_1 + \lambda^2 A_2$ has the following properties:
\begin{enumerate}
\item $A_0$ is a self-adjoint positive operator with dense domain $D$,
\item $A_0^{-1/2} A_1$ and $A_1 A_0^{-1/2}$ can be extended to bounded operators on $L^2(\O)$,
\item $A_0^{-1/2}$ is in some Schatten class $\mathscr{C}^p(L^2(\O), L^2(\O))$, $0<p<\infty$.
\end{enumerate}
These would imply completeness after some calculations. The main trick is to reduce the consideration of $T(\lambda)^{-1}$ to a matrix-valued resolvent $(\A - \lambda)^{-1}$. This reduction is the same kind as used to reduce a high-order ordinary differential equation into a first order matrix valued ordinary differential equation. Robert and Lai \cite[Prop 1.3 proof]{LaiRobertNL} attributed this trick to Agmon and Niremberg. Then use the theory for the resolvents of non-self-adjoint operators, for example Agmon \cite[Sec. 16]{AgmonLecturesOnElPDE} or Dunford and Schwartz \cite[XI.9.31]{DunfordSchwartz2}.

The details of the above deduction seemed to contain some redundancy. The first step was to simplify the completeness proofs from \cite{AgmonLecturesOnElPDE} and \cite{DunfordSchwartz2} to our case. Both of them were formulated for resolvents but used the Phragm\'en-Lindel\"of principle as a punchline. The most nontrivial part of those proofs is the so called \emph{Carleman's inequality}, which is a resolvent estimate for non-normal operators. After identifying and isolating this, the structure of the proof became very simple: It requires the analyticity of $T(\lambda)$, the meromorphicity of $T(\lambda)^{-1}$, its boundedness on some rays, and a limit to its growth on bigger and bigger circles.

Meromorphicity and the boundedness on rays follows like in the original way of proving the result, i.e. using Agravonich and Vishik \cite{AgranovichVishik} and the analytic Fredholm theorem. Actually, according to Theorem 5 in \cite{Agranovich1976}, this would be enough. We decide to show also other ways of proving completeness to keep the text more self-contained, and because using Nevanlinna theory may open doors to considering more general families of operators.

The last thing to do after showing meromorphicity and boundedness on some rays is to estimate the growth of $\norm{T(\lambda)^{-1}}$ on big circles. There are at least two ways of doing that. The ``black box'' way is to use Olavi Nevanlinna's $\Tone$, more specifically Theorem 3.2 in \cite{nevanlinna}. The other way is more concrete and gives a better result, but is longer and only works when $T(\lambda)$ is a polynomial. Basically, a linearization allows us to consider a resolvent $(\A - \lambda)^{-1}$ instead of the meromorphic $T(\lambda)^{-1}$.

We will have an estimate $\norm{T(\lambda)^{-1}} \leq \norm{(\A - \lambda)^{-1}}$, from which we can continue by using Carleman's inequality for resolvents. This inequality is of the form
\begin{equation}
\norm{ \varphi_K (\Id - K)^{-1}} \leq \exp C\norm{K}_{\mathscr{C}^p}^p,
\end{equation}
where $K$ is Schatten $p$ and $\varphi_K$ a so-called ``determinant''. In practice $K$ and $\varphi_K$ will depend analytically on $\lambda$ and the job of $\varphi_K$ is to remove the singularities of $(\Id - K(\lambda))^{-1}$. The usual way to prove Carleman's resolvent inequality is to reduce to a finite-dimensional case since $K$ is Schatten. The function $\varphi_K(\lambda)$ is built using a Weierstrass product. We will have
\begin{equation}
\begin{split}
K(\lambda) &= (\lambda - \lambda')(\A - \lambda')^{-1},\\
\varphi_K(\lambda) &= \prod_{j=0}^\infty \left( 1 - \frac{\lambda - \lambda'}{\lambda_j - \lambda'} \right) \exp \left( \frac{\lambda - \lambda'}{\lambda_j - \lambda'}  + \cdots + \frac{1}{k-1} \left( \frac{\lambda - \lambda'}{\lambda_j - \lambda'}  \right)^{k-1} \right),
\end{split}
\end{equation}
where $\lambda'$ is in the resolvent set of $\A$ and $\lambda_j$ are its eigenvalues, or equivalently the poles of $T(\lambda)^{-1}$. The operator $K$ is Schatten $p$ so $\sum \abs{\frac{\lambda - \lambda'}{\lambda_j - \lambda'}}^p < \infty$ for all $\lambda$. Hence the Weierstrass product converges to an entire function of order $p$. This also implies that $\abs{\varphi_K(\lambda)} > e^{-\abs{\lambda}^{p+\varepsilon}}$ on some larger and larger circles.

We finish by
\begin{multline}
\norm{T(\lambda)^{-1}} \leq \norm{(\A - \lambda)^{-1}} = \norm{ (\A - \lambda')^{-1} \frac{1}{\varphi_K(\lambda)} \varphi_K(\lambda) (\Id - K)^{-1} } \\
\leq \norm{(\A - \lambda')^{-1}} \frac{1}{\abs{\varphi_K(\lambda)}} \norm{\varphi_K(\lambda) (\Id - K(\lambda))^{-1}} \\
\leq C_{\A, \lambda'} e^{\abs{\lambda}^{p + \varepsilon}} e^{C\abs{\lambda}^p} \leq e^{C\abs{\lambda}^{p + \varepsilon}}
\end{multline}
on some circles $\abs{\lambda} = r_j$, $r_j \to \infty$.

\medskip
How does all of this combine into a proof? Start with $g \in L^2(\O)$ such that $(u,g) = 0$ for all generalized eigenstates $u$ of our operator $T(\lambda)$ with range $L^2(\O)$. To prove density of the span of generalized eigenstates, it's enough to show that the previous implies that $g = 0$, or in other words, that $(v,g) = 0$ for $v$ in a dense subset of $L^2(\O)$. Because $T(\lambda)$ will be invertible for some $\lambda'$, it is enough to show that 
\begin{equation}
w(\lambda) := \big( T(\lambda)^{-1}f, g \big) = 0
\end{equation}
for all $\lambda \in \C$ and $f \in L^2(\O)$.

The definition of generalized eigenstates will imply that $w$ is entire on $\C$. In particular the principal parts of all the Laurent series expansions of $T(\lambda)^{-1}$ will map $L^2(\O)$ into the span of the generalized eigenstates.

Boundedness of $T(\lambda)^{-1}$ on some rays and its growth rate on some circles will allow us to use the Phragm\'en-Lindel\"of principle to see that $w$ is actually bounded on $\C$. Hence it is a constant. The constant is zero since $T(\lambda)^{-1} \to 0$ as $\lambda \to \infty$ along some rays.

\bigskip
We will still give some more background information about the different tools used in the proof.

\subsection{Generalized eigenstates}
Let $T(\lambda)$ be an analytic family of operators in a Hilbert space $H$, $\lambda_0 \in \C$ a point of non-invertibility and consider the Taylor expansion $T(\lambda) = \sum B_m (\lambda - \lambda_0)^m$. A function $u$ will be called a generalized eigenstate of $T$ associated with the generalized eigenvalue $\lambda_0$ if there is $k\in\N$ and $u_0, u_1, \ldots, u_{k-1} \in H$ such that
\begin{equation}
\begin{split}
&B_0 u_0 = 0,\\
&B_1u_0 + B_0 u_1 = 0, \\
&B_2u_0 + B_1 u_1 + B_0 u_2 = 0,\\
&\quad\qquad\vdots\\
&B_ku_0 + B_{k-1}u_1 + \cdots + B_1 u_{k-1} + B_0 u = 0.
\end{split}
\end{equation}
As far as we know, the definition goes back to Keldysh \cite{KeldyshOrig,KeldyshENG}. One usually calls $v$ a generalized eigenvector of a matrix $M$ if $(M - \lambda)^k v = 0$ for some natural number $k$. Using the same name for both definitions is not a coincidence. 

Consider the case of $T(\lambda) = A_0 + \lambda A_1 + \lambda^2 A_2$ with $A_2$ invertible. Then the equation $T(\lambda) u = 0$ can be linearized to
\begin{equation}
(\A - \lambda) \bm u \\ v \sm = 0, \qquad \A = \bm 0 & A_2^{-1} \\ -A_0 & -A_1 A_2^{-1} \sm.
\end{equation}
To ease the notation assume that $\lambda = 0$ is a point of non-invertibility for $T$. This will make $B_j = A_j$ for all $j$. Consider the generalized eigenvectors of the matrix $\A$ corresponding to the eigenvalue $0$. We have
\begin{enumerate}
\item $\bm u_0 \\ v_0 \sm$ is an eigenvector $\Leftrightarrow \A \bm u_0 \\ v_0 \sm = 0 \Leftrightarrow A_0 u_0 = 0$ and $v_0 = 0$
\item $\A^2 \bm u_1 \\ v_1 \sm = 0 \Leftrightarrow \bm u_0 \\ v_0 \sm := \A \bm u_1 \\ v_1 \sm = \bm A_2^{-1} v_1 \\ -A_0 u_1 - A_1 A_2^{-1} v_1 \sm$ is an eigenvector $\Leftrightarrow u_0 = A_2^{-1} v_1$, $v_0 + A_0 u_1 + A_1 A_2^{-1} v_1 = 0$, $v_0 = 0$, $A_0 u_0 = 0$, which is equivalent to $A_0 u_0 = 0$, $A_1 u_0 + A_0 u_1 = 0$, $v_0 = 0$ and $v_1 = A_2 u_0$.
\end{enumerate}
Continuing similarly always defining $\bm u_j \\ v_j \sm = \A \bm u_{j+1} \\ v_{j+1} \sm$ we get the equation array
\begin{equation}
\begin{split}
v_0 = 0 \quad & A_0 u_0 = 0\\
v_1 = A_2 u_0 \quad & A_1 u_0 + A_0 u_1 = 0 \\
v_2 = A_2 u_1 \quad & A_2 u_0 + A_1 u_1 + A_0 u_2= 0 \\
 \vdots & \\
v_k = A_2 u_{k-1} \quad  & A_2 u_{k-2} + A_1 u_{k-1} + A_0 u_k = 0\\
 \vdots &
\end{split}
\end{equation}
which is just the definition for the generalized eigenstate of $T(\lambda)$ because we had $B_j = A_j$ when $\lambda = 0$ is the point of non-invertibility.

Both definitions have been used in different articles proving density results. For the case of $(M - \lambda)^k$ see for example Section 16 of \cite{AgmonLecturesOnElPDE}, the definition right before XI.6.29 in \cite{DunfordSchwartz2} and then XI.6.29, XI.9.29 and their corollaries. For the case of generalized eigenstate of an analytic family see for example Keldysh \cite{KeldyshOrig,KeldyshENG}, the introduction in Linden \cite{Linden} and Chapter II in Markus \cite{Markus}. When consulting Keldysh, note that his operator $A(\lambda)$ corresponds to our $(T(\lambda') - T(\lambda))T(\lambda')^{-1}$, where $T(\lambda')$ is assumed invertible and the difference at the two points will be a compact operator.

\subsection{Boundedness of $T(\lambda)^{-1}$ on rays}
The goal here is to be able to use the Phragm\'en-Lindel\"of principle on the function $(T(\lambda)^{-1}f,g)$, where $f$ and $g$ will be such that the map is entire. Hence we will have to show that $\norm{T(\lambda)^{-1}}$ is bounded on some rays.

It will be a straightforward application of Agranovich and Vishik \cite{AgranovichVishik} because $T(\lambda)$ will be \emph{elliptic with parameter} when taking into account the boundary values of the elements of its domain. We will not go into the details of their proof as their article is very clearly written. Their proof considers first a few model cases, namely $\O = \R^n$ and $\O = \R^{n-1} \times \R_+$, then they straighten the boundary and use a fine partition of unity to freeze the coefficients of the operator. The model cases and some semiclassical Sobolev space estimates then give the result:
\begin{theorem*}
Let $D = H^4 \cap H^2_0(\Omega)$ with norm $\norm{u}_D^2 = \norm{u}_{H^4}^2 + \abs{\lambda}^4 \norm{u}_{L^2}^2$. Assume that $T(\lambda) : D \to L^2(\O)$ is elliptic with parameter in a closed cone $Q$. Then there is $r_0 > 0$ such that if $\abs{\lambda}>r_0$ and $\lambda \in Q$, then $T(\lambda)^{-1}$ is bounded $L^2(\O) \to D$.
\end{theorem*}

For our case of interior transmission eigenvalues of either the Helmholtz or the Schr\"odinger equation, $T(\lambda)$ will be elliptic with parameter in every cone not touching $\R_-$ and vertex at the origin. The analytic Fredholm theorem applied to the Fredholm operator
\begin{equation}
T(\lambda) T(\lambda')^{-1} = \Id - (T(\lambda') - T(\lambda))T(\lambda')^{-1}
\end{equation}
where $T(\lambda')$ is invertible will imply that $T(\lambda)^{-1}$ is meromorphic. Hence the inverse exists everywhere except on a discrete set of poles.

The semiclassical norm-estimate from Agranovich and Vishik will give us $\norm{T(\lambda)^{-1}}_{L^2(\O) \to L^2(\O)} \to 0$ when $\abs{\lambda} \to \infty$ along rays not pointing towards $-\infty$. Hence $\norm{T(\lambda)^{-1}}$ will be bounded on a sufficiently large set of rays.

\smallskip
The only remaining matter is to define ellipticity with parameter. For that, $T(\lambda)$ must be a polynomial. Moreover some conditions are imposed on the symbols of $T$ and the boundary operators used to define its domain. Since the final estimate for $T(\lambda)^{-1}$ will take the parameter into account, the symbols have to be defined carefully. The symbol of $T(\lambda)$ is defined by: transform $\nabla \mapsto i \xi$, $\xi \in \R^n$,  and then consider the resulting operator as a polynomial\footnote{Actually there is no unique way of doing this. We may as well consider it as a polynomial in $(\xi, \sqrt{\lambda})$, however not all choices will give ellipticity with parameter.} in $(\xi, \lambda)$. The principal symbol will be the sum of all highest order terms. The same trick will be done for the operators defining the boundary conditions. If $\tau > 0$, 
\begin{equation}
T(\lambda) = \sum_{\abs{\alpha} + l \leq m} c_{\alpha, l}(x) \lambda^l \frac{d^{\abs{\alpha}}}{dx^\alpha},
\end{equation}
and $M = \max \{ \abs{\alpha} + \tau l \mid c_{\alpha, l} \neq 0\}$, then the principal symbol with weight $\tau$ is
\begin{equation}
T_0(x,\xi,\lambda) = \sum_{\abs{\alpha} + \tau l = M} c_{\alpha, l}(x) \lambda^l(i \xi)^\alpha.
\end{equation}
In this article we will have $\tau = 2$ when considering the Helmholtz and Schr\"odinger inner transmission problems.

Let $A(\xi, \lambda)$ be the principal symbol of the operator and $B_j(\xi, \lambda)$, $j=0,1,\ldots$ the principal symbols of the boundary operators. Then the requirements are
\begin{enumerate}[I.]
\item $A(\xi, \lambda) \neq 0$ when $(\xi, \lambda) \neq 0$,
\item After rotating and translating the coordinates such that $0 \in \partial \O$ and $-e_n$ is the outer normal, we require that the problem
\begin{equation}
\begin{cases}
A\big( (\xi', -i \tfrac{d}{dt}), \lambda\big) v (t) = 0, & t >0 \\
B_j \big( (\xi', -i \tfrac{d}{dt}), \lambda\big) v (t) = h_j, & t=0
\end{cases}
\end{equation}
has a single solution in the space of functions vanishing at infinity for each choice of the numbers $h_j$ as long as $(\xi', \lambda) \neq 0$.
\end{enumerate}
The second condition is a generalization of the well known Shapiro-Lopatinsky condition. See Shapiro \cite{Shapiro}, Lopatinsky \cite{Lopatinsky} and the introduction in \cite{KatyaSHAPLOP}.

\subsection{Growth of $T(\lambda)^{-1}$ on $\abs{\lambda} = R$}
Here the goal is again to be able to use the Phragm\'en-Linderl\"of theorem on the entire function $(T(\lambda)^{-1}f, g)$. For that we want to show that it is of order $p$, that is
\begin{equation}
\abs{ (T(\lambda)^{-1} f, g)} \leq e^{C \abs{\lambda}^p}
\end{equation}
We know from before that $T(\lambda)^{-1}$ is meromorphic, so the poles will cause problems. How can we avoid them? Basically we will say that $T(\lambda)^{-1}$ is a meromorphic family of operators of order $p$.

There are at least three ways to avoid the poles blowing up the estimate. The simplest one is to have the estimate only on bigger and bigger circles $\abs{\lambda} = r_j$, $r_j \to \infty$, none of them being too close to the poles. Another way would be to have it true whenever $\abs{\lambda - \lambda_j} > \abs{\lambda_j}^{-p-\varepsilon}$ for all the poles $\lambda_j$. See for example Titchmarsh \cite[8.71]{Titch}. The third way seems to be the most natural one for considering the growth of meromorphic functions. It uses $\Tinf$, the \emph{Nevanlinna characteristic} from function theory. We would have to show that
\begin{equation}
\Tinf (r, T(\lambda)^{-1}) \leq C r^p.
\end{equation}
These are all three sides of the same coin. It seems to relate to Nevanlinna theory, starting from Weierstrass products. The underlining idea here is that if $\lambda_j$ are the poles of $T(\lambda)^{-1}$ counting multiplicities, then we will have
\begin{equation}
\sum \abs{\lambda_j}^{-p} < \infty,
\end{equation}
and this is a very strong assumption in function theory. An entire function $f$ whose zeros satisfy such an inequality is necessarily of order $p$ and $1/f$ can be estimated from below by $\exp R_j^{-p-\varepsilon}$ on bigger and bigger circles of radius $R_j$. Moreover given such $\lambda_j$ we can construct an entire function of order $p$ vanishing only at those points.

Unfortunately the previous deductions can't be generalized for operator valued meromorphic functions so easily. For those readers who are more familiar with Nevanlinna theory we suggest to read about the characteristic $\Tone$ from \cite{nevanlinna}, especially Theorem 3.2. It will allow us to estimate the growth of the inverse of an analytic family of Fredholm operators if we know some simple properties of that family. For the other readers, we will also use a more functional analytic approach. We will construct the entire function $\varphi(\lambda)$ vanishing at the poles, and then use Carleman's resolvent inequality to get an upper bound for $\varphi(\lambda) T(\lambda)^{-1}$. This will lead to the desired estimate.

\newpage
\section{Completeness theorem}
\label{meroSect}
The theorem of this section has appeared in various sources in different forms. See for example Chapter II in Keldysh \cite{KeldyshENG}, XI.6.29 and XI.9.29 with related corollaries in Dunford and Schwartz \cite{DunfordSchwartz2}, and also Section 16 in Agmon \cite{AgmonLecturesOnElPDE}. We shall prove a version which suits well our needs.

In the next theorem $H$ is a fixed Hilbert space and $D \subset H$ a dense normed subspace. We will have the following three assumptions in this section:
\begin{enumerate}
\item Let $T(\lambda) : \C \to \mathscr{L}(D,H)$ be analytic.
\item Assume that $T(\lambda)^{-1}$ is meromorphic $\C \to \mathscr{L}(H,H)$, bounded on a number of rays partitioning $\C$ into cones of angle less than $\pi/\beta$ and that $\norm{T(\lambda)^{-1}} \to 0$ on some sequence in $\C$.
\item Moreover assume that if $f, g \in H$ are such that $\lambda \mapsto (T(\lambda)^{-1}f,g)$ is entire, then
\begin{equation}
\abs{(T(\lambda)^{-1}f, g)} \leq \norm{f}\norm{g}e^{C\abs{\lambda}^\beta}
\end{equation}
for sufficiently large $\lambda$.
\end{enumerate}

\begin{definition}
\label{genEigenDef}
For $\lambda_0 \in \C$, $m \in \N$, we write $B_m = \frac{1}{m!} \frac{d^m}{d\lambda^m} T(\lambda)_{|\lambda=\lambda_0}$. Now we say $u \in \sp(\lambda_0)$ if there is $k\in\N$ and non-zero $u_0, u_1, \ldots, u_{k-1} \in H$ such that
\begin{equation}
\begin{split}
&B_0 u_0 = 0,\\
&B_1u_0 + B_0 u_1 = 0, \\
&B_2u_0 + B_1 u_1 + B_0 u_2 = 0,\\
&\quad\qquad\vdots\\
&B_ku_0 + B_{k-1}u_1 + \cdots + B_1 u_{k-1} + B_0 u = 0.
\end{split}
\end{equation}
$\sp(\lambda_0)$ is the set of \emph{generalized eigenstates} related to the \emph{singular value} $\lambda_0$.
\end{definition}

\begin{theorem}
\label{denseThm}
Let $\lambda_0, \lambda_1, \ldots$ be the poles of $T(\lambda)^{-1}$. Then $\operatorname{span} \cup_{j=0}^\infty \sp(\lambda_j)$ is dense in $H$.
\end{theorem}
\begin{proof}
We will show that if $g \in H$ and $(u,g) = 0$ for all $u \in \cup_{j=0}^\infty \sp(\lambda_j)$, then $g=0$. This follows if $(v,g)=0$ for all $v \in D$ since the domain $D$ of $T(\lambda)$ is dense in $H$. It is enough to show that there is some $\lambda$ not a pole, such that
\begin{equation}
w(\lambda) := (T(\lambda)^{-1}f, g) = 0
\end{equation}
for all $f \in H$. We will show that $w$ is analytic and $w(\lambda) = 0$ for all $\lambda$.

If $\lambda'$ is not a pole, then $w(\lambda)$ is analytic near $\lambda'$. So let $\lambda'$ be a pole, for example $\lambda' = \lambda_0$. If $\abs{\lambda-\lambda_0}$ is small, we have the power series expansions
\begin{equation}
T(\lambda) = \sum_{m=0}^\infty B_m (\lambda-\lambda_0)^m, \quad T(\lambda)^{-1} = \sum_{n=-N}^\infty C_n (\lambda-\lambda_0)^n.
\end{equation}
Hence we get
\begin{multline}
\operatorname{Id} = T(\lambda) T(\lambda)^{-1} = \left(\sum_{m=0}^\infty B_m (\lambda-\lambda_0)^m \right) \left( \sum_{n=-N}^\infty C_n (\lambda-\lambda_0)^n \right)\\
= \left(\sum_{m=0}^\infty B_m (\lambda-\lambda_0)^m \right) \left( (\lambda-\lambda_0)^{-N} \sum_{n=0}^\infty C_{n-N} (\lambda-\lambda_0)^n \right) \\
= \sum_{k=0}^{N-1} \left( \sum_{j=0}^k B_{k-j}C_{j-N}\right) (\lambda-\lambda_0)^{k-N} + \sum_{k=N}^\infty \left( \sum_{j=0}^k B_{k-j}C_{j-N}\right) (\lambda-\lambda_0)^{k-N}.
\end{multline}
Both sides are meromorphic, and the left hand side is analytic. Hence the right hand side's power series must represent an analytic function. This means that
\begin{equation}
 \sum_{j=0}^k B_{k-j}C_{j-N} = 0, \qquad k = 0, 1, \ldots, N-1.
\end{equation}
Writing this out more explicitly gives
\begin{equation}
\begin{split}
&B_0 C_{-N} = 0,\\
&B_1C_{-N} + B_0 C_{1-N} = 0, \\
&B_2 C_{-N} + B_1 C_{1-N} + B_0 C_{2-N} = 0,\\
&\quad\qquad\vdots\\
&B_k C_{-N} + B_{k-1} C_{1-N} + \cdots + B_1 C_{-2} + B_0 C_{-1} = 0,
\end{split}
\end{equation}
which just says that $C_n : H \to \sp(\lambda_0)$ for $n = -1, -2, \ldots, -N$.

Since we have $g \perp \sp(\lambda_0)$, we see that
\begin{equation}
w(\lambda) = \sum_{n=-N}^\infty (C_n f, g) (\lambda-\lambda_0)^n = \sum_{n=0}^\infty (C_n f, g) (\lambda - \lambda_0)^n.
\end{equation}
Thus $w(\lambda) = (T(\lambda)^{-1}f, g)$ is analytic in a neighborhood of $\lambda_0$. By doing the same deductions for all the other poles too, we see that $w$ is analytic in the whole $\C$.

\smallskip
Partition the complex plane into a finite number of cones, all with vertex at a single common point, with angle less than $\pi / \beta$ and none of them having a pole on their boundary. By Assumption 2 we have
\begin{equation}
\norm{f}_H^{-1} \norm{g}_H^{-1}\abs{w(\lambda)} \leq \norm{T(\lambda)^{-1}} \leq C < \infty
\end{equation}
on the boundary rays, and by Assumption 3
\begin{equation}
\norm{f}_H^{-1} \norm{g}_H^{-1}\abs{w(\lambda)} \leq  e^{C \abs{\lambda}^\beta}
\end{equation}
when $\abs{\lambda}$ is large. The Phragm\'en-Lindel\"of principle (e.g. \cite[Thm 5.61]{Titch}) tells us now that $w(\lambda)$ is bounded on the whole $\C$. By Liouville's theorem, $w(\lambda)$ is a constant. This constant is zero by the last part of Assumption 2.
\end{proof}

\newpage
\section{Schatten class embedding}
\label{schattenSect}
We add here a proof of the fact that $H^2_0(\O)$ embeds compactly into $L^2(\O)$. Actually the embedding is of Schatten class $p > n/2$. See for example Definition XI.9.1 in \cite{DunfordSchwartz2}. We will need this fact when using analytic Fredholm theory and when proving that $T(\lambda)^{-1}$ is meromorphic of finite order.

\begin{lemma}
\label{schattenIncl}
Let $\O \subset \R^n$ be a bounded domain. Then the inclusion
\begin{equation}
\iota : H^2_0(\O) \hookrightarrow L^2(\O)
\end{equation}
is compact. Moreover it is Schatten $\mathscr{C}^p(H^2_0(\O), L^2(\O))$ for $p>n/2$.
\end{lemma}
\begin{proof}
Let $\mathbb{T}^n$ be the $n$-dimensional torus. Let $E : H^2_0(\O) \to H^2(\mathbb{T})$ be given as follows: extend $u \in H^2_0(\O)$ as zero to a big enough hypercube which is identified as $\mathbb{T}^n$ by extending $u$ then periodically.

Let $P_\O : L^2(\mathbb{T}^n) \to L^2(\O)$ be given by projecting to a periodic function on $\R^n$ and then multiplying by the characteristic function of $\O$. Let $I : H^2(\mathbb{T}^n) \hookrightarrow L^2(\mathbb{T}^n)$. Now
\begin{equation}
\iota = P_\O \circ I \circ E,
\end{equation}
where $E$ and $P_\O$ are bounded. It is enough to prove that $I$ is compact and Schatten $p$ for $p>n/2$.

We have $u \in H^2(\mathbb{T}^n)$ if and only if $(1 + \abs{\xi}^2)\hat{u}(\xi) \in \ell^2(\Z^n)$ and $u \in L^2(\mathbb{T}^n)$, with equivalent norms. Let
\begin{equation}
J : \begin{cases}L^2(\mathbb{T}^n) \to L^2(\mathbb{T}^n) \\ u \mapsto \sum_{\xi \in \Z^n} \frac{\langle u(y), e^{2\pi i y\cdot \xi}\rangle}{1 + \abs{\xi}^2} e^{2\pi i x \cdot \xi} \end{cases}. 
\end{equation}
We have $I = J \circ (J^{-1})_{|H^2(\mathbb{T}^n)}$, where $J^{-1}$ is bounded $H^2(\mathbb{T}^n) \to L^2(\mathbb{T}^n)$. Hence it is enough to prove that $J$ is compact and Schatten $p$ for $p>n/2$.

The sequence $(e^{2\pi i x \cdot \xi} \mid \xi\in\Z^n)$ is a Hiblert basis for $L^2(\mathbb{T}^n)$ and $(\frac{1}{1+\abs{\xi}^2} \mid \xi \in \Z^n)$ can be ordered to a decreasing sequence $\gamma_j$, with each $\gamma_j$ having finite multiplicity. Hence $J$ is compact and its eigenvalue-eigenstate pairs are
\begin{equation}
\Big( \frac{1}{1+\abs{\xi}^2}, e^{2\pi i x\cdot \xi} \Big)
\end{equation}
for $\xi \in \Z^n$.

Because $\frac{1}{1+\abs{\xi}^2} > 0$, $J$ is self-adjoint and positive. Hence its singular values are its eigenvalues, an so
\begin{multline}
\norm{J}_{\mathscr{C}^p(L^2(\mathbb{T}^n), L^2(\mathbb{T}^n))} = \sum_{\xi\in\Z^n} \frac{1}{(1 + \abs{\xi}^2)^p} \\
\leq \sum_{\xi\in\Z^n} (1+n)^p \int_{[\xi_1, \xi_1+1[ \times \cdots \times [\xi_n, \xi_n+1[} \frac{dm(x)}{(1+\abs{x}^2)^p} = (1+n)^p \int_{\R^n} \frac{dm(x)}{(1+\abs{x}^2)^p} \\
= (1+n)^p \sigma(\mathbb{S}^{n-1}) \int_0^\infty \frac{r^{n-1}dr}{(1+r^2)^p} < \infty
\end{multline}
when $2p - n + 1 > 1$, i.e. $p > n/2$.
\end{proof}

\begin{remark}
\label{schattenRem}
The same proof shows that $\iota' : H^{k+2}(\O) \to H^k(\O)$ is Schatten $p$ for $p > n/2$ if $\O$ has an extension operator $E:H^{k+2}(\O) \to H^{k+2}(\R^n)$. Now instead of extending $u$ as zero, we extend it as $\phi E u$, where $\phi \in C^\infty_0(\R^n)$ is constant one near $\O$.
\end{remark}

\newpage
\section{Invertibility and boundedness on rays}
\label{elliptParamSect}
We will start by proving Conditions I and II from page 74 of Agranovich and Vishik \cite{AgranovichVishik}. This means that the operator under consideration is parameter-elliptic. The goal is to prove that the maps
\begin{equation}
\begin{split}
T_H : \,&H^4\cap H^2_0(\O) \to L^2(\O)\\
& u \mapsto  \Delta q \Delta u - \lambda(\Delta q + q\Delta + \Delta) u + \lambda^2(1+q) u \\
T_S : \,&H^4\cap H^2_0(\O) \to L^2(\O)\\
& u \mapsto  \Delta q \Delta u + \Delta - \lambda(\Delta q + q\Delta + 1) u + \lambda^2q u
\end{split}
\end{equation}
corresponding to the interior transmission eigenvalue problems of the Helmholtz and Schr\"odinger equations have a bounded inverse when $\lambda \in \C\setminus \R_-$ is large enough, and that the inverse vanishes at infinity. Basically we want to prove that Assumption 2 of Section \ref{meroSect} is satisfied.

\bigskip
First write $T_H$ and $T_S$ in the form $\sum c_\alpha(x) D^\alpha$, where $\nabla = i D$:
\begin{multline}
T_H(x,D,\lambda) = q(x) (D\cdot D)^2 - 2 i \nabla q(x) \cdot D (D\cdot D) - \Delta q(x) (D\cdot D) \\
- \lambda\big( -(1 + 2q(x)) (D \cdot D) + 2i\nabla q(x) \cdot D + \Delta q(x)\big) + \lambda^2 (1 + q(x))
\end{multline}
\begin{multline}
T_S(x,D,\lambda) = q(x) (D\cdot D)^2 - 2 i \nabla q(x) \cdot D (D\cdot D) - (1 +\Delta q(x)) (D\cdot D) \\
- \lambda\big( - 2q(x) (D \cdot D) + 2i\nabla q(x) \cdot D + \Delta q(x) + 1\big) + \lambda^2 q(x)
\end{multline}
Along the differential operators we will use the boundary operators
\begin{eqnarray}
&&B_1(x,D,\lambda) = 1,\\
&&B_2(x,D,\lambda) = i \eta(x) \cdot D,
\end{eqnarray}
where $\eta(x)$ is the outer boundary normal vector at $x \in \partial\O$.

These will allow us to have
\begin{definition}
The \emph{principal symbols} are defined as
\begin{equation}
\begin{split}
T_{H0} (x,\xi,\lambda) &= q(x) \abs{\xi}^4 + \lambda(1+2q(x)) \abs{\xi}^2 + \lambda^2(1+q(x)),\\
T_{S0} (x,\xi,\lambda) &= q(x) \abs{\xi}^4 + 2 \lambda q(x) \abs{\xi}^2 + \lambda^2 q(x),\\
B_{10} (x,\xi,\lambda) &= 1,\\
B_{20} (x,\xi,\lambda) &= i \eta(x) \cdot \xi.
\end{split}
\end{equation}
\end{definition}
\begin{remark}
Note that these are gotten from the symbols of $T_H$, $T_S$, $B_1$ and $B_2$ by taking the highest order terms while considering $(\xi,\lambda^{1/2})$ as the variables.
\end{remark}

\begin{proposition}
\label{cond1prop}
Let $\O \subset \R^n$ be a bounded smooth domain and $q \in C^\infty(\overline{\O}, \R)$ be positive and bounded away from zero and infinity. Then Condition I \cite[p. 74]{AgranovichVishik} holds for $T_H$ and $T_S$. That is, for $x\in\overline{\O}$, $\lambda \in \C \setminus \R_-$, $\xi \in \R^n$ with $\abs{\xi} + \abs{\lambda} \neq 0$ we have $T_{H0} (x,\xi,\lambda) \neq 0$ and $T_{S0}(x,\xi,\lambda) \neq 0$.
\end{proposition}
\begin{proof}
Let's first prove it for the Schr\"odinger case.  We have
\begin{equation}
T_{S0} = q(x) \left( \abs{\xi}^2 + \lambda \right)^2 = 0
\end{equation}
if and only if $\abs{\xi}^2 + \lambda = 0$. The latter is impossible when $\lambda \in \C \setminus \R_-$ and $\abs{\xi} + \abs{\lambda} \neq 0$.

For the Helmholtz operator we have two cases:
\begin{flushleft}
Case 1: $\lambda = 0$. Now $\abs{\xi} \neq 0$ so $T_{H0}(x,\xi,0) = 0 \Leftrightarrow q(x) = 0$ which is not true.

Case 2: $\lambda \neq 0$. Now $\xi \in \R^n$ is arbitrary. We have
\begin{equation}
\begin{split}
T_{H0}(x,\xi,\lambda) &= 0 \\
\Leftrightarrow \lambda &= \frac{-(1+2q(x)) \abs{\xi}^2 \pm \sqrt{(1+2q(x))^2 \abs{\xi}^4 - 4q(x)(1+q(x))\abs{\xi}^4}}{2q(x)} \\
&= \frac{-1 - 2q(x) \pm \sqrt{1 + 4q(x) + 4q(x)^2 - 4q(x) - 4q(x)^2}}{2q(x)} \abs{\xi}^2 \\
&= \frac{-1-2q(x) \pm 1}{2q(x)} \abs{\xi}^2 \in \R_-,
\end{split}
\end{equation}
which is impossible since $\lambda \in \C\setminus \R_-$.
\end{flushleft}
\end{proof}

\medskip
We have to do some preparations to prove Condition II. It is formulated using a \emph{coordinate system connected with the point $x_0 \in \partial \O$}. These coordinates are defined by translating $x_0$ to the origin and then rotating so that the inner normal vector $-\eta(x_0)$ maps to the vector $e_n = (0,\ldots, 0, 1)$. See \cite[p. 63]{AgranovichVishik}. Let $y = \Phi(x)$ denote these new coordinates and $\tilde{T_{H0}}$, $\tilde{T_{S0}}$, $\tilde{B_{10}}$ and $\tilde{B_{20}}$ the principal symbols of the differential and boundary operators in these new coordinates. We get\begin{equation}
\begin{split}
\tilde{T_{H0}}(y, D_y, \lambda) &= q(\Phi^{-1}(y)) (D_y \cdot D_y)^2 + \lambda(1 + 2q(\Phi^{-1}(y))) D_y \cdot D_y \\
&\qquad + \lambda^2(1+q(\Phi^{-1}(y))), \\
\tilde{T_{S0}}(y, D_y, \lambda) &= q(\Phi^{-1}(y)) (D_y \cdot D_y)^2 + 2 \lambda q(\Phi^{-1}(y))) D_y \cdot D_y + \lambda^2 q(\Phi^{-1}(y)), \\
\tilde{B_{10}}(y, D_y, \lambda) &= 1, \quad \tilde{B_{20}}(y, D_y, \lambda) = - \tfrac{d}{dy_n}
\end{split}
\end{equation}
because the Laplacian $\Delta = - D\cdot D$ is invariant under rigid motions and $B_{20}$ is just the boundary normal derivative.

\begin{proposition}
\label{cond2prop}
Let $\O \subset \R^n$ be a bounded smooth domain and $q \in C^\infty(\overline{\O}, \R)$ be positive and bounded away from zero and infinity. Then Condition II \cite[p. 74]{AgranovichVishik} holds for $T_H$ and $T_S$. That is, let $x_0\in\partial\O$, $\lambda \in \C \setminus \R_-$ and $\xi' \in \R^{n-1}$ such that $\abs{\xi'} + \abs{\lambda} \neq 0$. Then in coordinate systems connected with $x_0$, the ordinary differential equations
\begin{equation}
\begin{cases}
\label{helmholtzODE}
\tilde{T_{H0}}(0, (\xi', \tfrac{1}{i} \tfrac{d}{dt}), \lambda) v(t) = 0, \quad t > 0\\
\tilde{B_{10}} (0, (\xi', \tfrac{1}{i} \tfrac{d}{dt}), \lambda) v(t) = h_1, \quad t= 0\\
\tilde{B_{20}} (0, (\xi', \tfrac{1}{i} \tfrac{d}{dt}), \lambda) v(t) = h_2, \quad t = 0
\end{cases}
\end{equation} 
and
\begin{equation}
\begin{cases}
\tilde{T_{S0}}(0, (\xi', \tfrac{1}{i} \tfrac{d}{dt}), \lambda) v(t) = 0, \quad t > 0\\
\tilde{B_{10}} (0, (\xi', \tfrac{1}{i} \tfrac{d}{dt}), \lambda) v(t) = h_1, \quad t= 0\\
\tilde{B_{20}} (0, (\xi', \tfrac{1}{i} \tfrac{d}{dt}), \lambda) v(t) = h_2, \quad t = 0
\end{cases}
\end{equation} 
have one and exactly one solution inside $\mathscr{S}(\R_+)$ for any $h_1, h_2 \in \C$.
\end{proposition}
\begin{proof}
Write $q = q(\Phi^{-1}(0)) = q(x_0)$. Consider the case of $T_{H0}$ first. The ordinary differential equation is then
\begin{equation}
\begin{cases}
q(\abs{\xi'}^2 - \tfrac{d^2}{dt^2})^2 v + \lambda(1+2q) (\abs{\xi'}^2 - \tfrac{d^2}{dt^2}) v + \lambda^2(1+q) v = 0, \quad t>0\\
v(0) = h_1, \quad v'(0) = -h_2.
\end{cases}
\end{equation}
We will solve it using the method of characteristic polynomial. Let 
\begin{equation}
P(r) = q (\abs{\xi'}^2 - r^2)^2 + \lambda(1+2q) (\abs{\xi'}^2 - r^2) + \lambda^2(1+q).
\end{equation}
Now
\begin{equation}
\begin{split}
P(r) = 0 \Leftrightarrow \abs{\xi'}^2 - r^2 &= \frac{-\lambda(1+2q) \pm \sqrt{\lambda^2 (1+2q)^2 - 4q\lambda^2(1+q)}}{2q} \\
&= \frac{-1-2q \pm \sqrt{1 + 4q + 4q^2 -4q - 4q^2}}{2q}\lambda \\
&= \frac{-1-2q \pm 1}{2q} \lambda = \begin{cases} -\lambda \\ -\lambda(1+1/q) \end{cases}.
\end{split}
\end{equation}
Let $r_1^2 = r_2^2 = \lambda + \abs{\xi'}^2$ and $r_3^2 = r_4^2 = \lambda(1+1/q) + \abs{\xi'}^2$ with $\Re r_1, \Re r_3 > 0$ and $\Re r_2, \Re r_4 < 0$. This is possible since $\lambda \notin \R_-$, $q \in \R_+$ and $\abs{\xi'} + \abs{\lambda} \neq 0$.

\smallskip
Case 1: $\lambda = 0$ and $\abs{\xi'} \neq 0$. Now we have $r_1 = r_3 = \abs{\xi'}$ and $r_2 = r_4 = -\abs{\xi'}$. They give the general solution
\begin{equation}
v(t) = (c_1 + c_3 t) e^{\abs{\xi'} t} + (c_2 + c_4 t) e^{-\abs{\xi'} t}.
\end{equation}
Because we are looking for $v$ vanishing at infinity, we must have $c_1 = c_3 = 0$. Now the boundary conditions reduce to
\begin{equation}
c_2 = v(0) = h_1, \qquad -\abs{\xi'}c_2 + c_4 = v'(0) = -h_2,
\end{equation}
whose unique solution is $c_2 = h_1$, $c_4 = \abs{\xi'}h_1 - h_2$.

\smallskip
Case 2: $\lambda \neq 0$ and $\xi' \in \R^{n-1}$ arbitrary. Now $r_1, r_2, r_3$ and $r_4$ are all different and hence the general solution is
\begin{equation}
v(t) = c_1 e^{r_1 t} + c_2 e^{r_2 t} + c_3 e^{r_3 t} + c_4 e^{r_4 t}.
\end{equation}
By the condition at infinity, we have similarly that $c_1 = c_3 = 0$. The boundary conditions are now
\begin{equation}
c_2 + c_4 = v(0) = h_1, \qquad r_2 c_2 + r_4 c_4 = v'(0) = -h_2.
\end{equation}
This has a unique solution since $\left\vert \begin{smallmatrix} 1 & 1\\ r_2 & r_4 \end{smallmatrix} \right\vert = r_4 - r_2 \neq 0$.

\medskip
Tackle the case of $T_{S0}$ now. The ordinary differential equation is
\begin{equation}
\begin{cases}
q \left( \abs{\xi'}^2 - \tfrac{d^2}{dt^2} + \lambda\right)^2 v  = 0, \quad t>0\\
v(0) = h_1, \quad v'(0) = -h_2.
\end{cases}
\end{equation}
We will solve it using the method of characteristic polynomial. Let 
\begin{equation}
P(r) = q (\abs{\xi'}^2 + \lambda - r^2)^2.
\end{equation}
Now
\begin{equation}
P(r) = 0 \Leftrightarrow r^2 = \abs{\xi'}^2 + \lambda,
\end{equation}
each of the roots having multiplicity two. Let $r_1 = r_3$ and $r_2 = r_4$ be the solutions with $\Re r_1, \Re r_3 > 0$ and $\Re r_2, \Re r_4 < 0$. This is possible since $\lambda \notin \R_-$ and $\abs{\xi'} + \abs{\lambda} \neq 0$. They give the general solution
\begin{equation}
v(t) = (c_1 + c_3 t) e^{r_1 t} + (c_2 + c_4 t) e^{r_2 t}.
\end{equation}
Because we are looking for $v$ vanishing at infinity, we must have $c_1 = c_3 = 0$. Now the boundary conditions reduce to
\begin{equation}
c_2 = v(0) = h_1, \qquad r_2 c_2 + c_4 = v'(0) = -h_2,
\end{equation}
whose unique solution is $c_2 = h_1$, $c_4 = -r_2 h_1 - h_2$.
\end{proof}

\bigskip
We can now prove invertibility. Note that the estimate will hold only for large $\lambda$ contained in a closed sector not touching $\R_-$. We will have to use some functional analysis to get the invertibility everywhere except for a discrete set of points.
\begin{proposition}
\label{invprop}
Let $\O \subset \R^n$ be a bounded smooth domain and $q \in C^\infty(\overline{\O}, \R)$ be positive and bounded away from zero and infinity. Let $Q \subset \C$ be a closed sector in the complex plane with vertex at the origin and not containing $\R_-$. Now there is $R \geq 0$ such that if $\lambda \in Q$, $\abs{\lambda} \geq R$ and $f\in L^2(\O)$, then there is a unique $u \in H^4\cap H^2_0(\O)$ such that $T_{H}(\lambda) u = f$. Moreover we have the estimate
\begin{equation}
\norm{u}_{H^4(\O)}^2 + \abs{\lambda}^4 \norm{u}_{L^2(\O)}^2 \leq C \norm{f}_{L^2(\O)}^2
\end{equation}
with $C$ independent of $u$, $f$ and $\lambda$. The claim holds for $T_{S}$ too.
\end{proposition}
\begin{proof}
We refer to the paper from Agranovich and Vishik \cite{AgranovichVishik}. Note that in that paper the authors write $q$ for the parameter, and in there it has the same weight as differentiation. Hence choose $q = \lambda^{1/2}$ when consulting their results.

Proposition \ref{cond1prop} and Proposition \ref{cond2prop} show that Condition I and Condition II in \cite[p. 74]{AgranovichVishik} are satisfied for both $T_{H}$ and $T_{S}$. Note that 
\begin{equation}
u \in H^4 \cap H^2_0(\O) \Leftrightarrow u \in H^4(\O) \text{ and } \operatorname{Tr} u = \operatorname{Tr} \eta(x) \cdot \nabla u = 0.
\end{equation}
Hence Theorem 5.1 in \cite[p. 84]{AgranovichVishik} gives existence and Theorem 4.1 in \cite[p. 75]{AgranovichVishik} gives uniqueness and the estimate. Note that they have hidden the parameter $\lambda^{1/2}$ into the semiclassical Sobolev space norms unlike us.
\end{proof}

\begin{theorem}
\label{merothm}
Let $\O \subset \R^n$ be a bounded smooth domain and $q \in C^\infty(\overline{\O}, \R)$ be positive and bounded away from zero and infinity and let $T$ denote either one of $T_{H}$ or $T_{S}$. Then $T(\lambda)^{-1}$ is meromorphic $\C \to \mathscr{L}(L^2(\O))$ and its poles are in a neighborhood of $\R_-$ of the form $\cup_{t>0} B(-t,r(t))$, where $r(t)t^{-1} \to 0$ at infinity. Moreover the inverse satisfies
\begin{equation}
\norm{T(\lambda)^{-1}} \leq C \abs{\lambda}^{-2}
\end{equation}
on all rays not touching its discrete set of poles and not having direction $(-1,0)$. Here $C$ depends on the ray.
\end{theorem}
\begin{proof}
To prove that the inverse is meromorphic, we will use the analytic Fredholm theorem. We will prove it again in Theorem \ref{inverseEstimate} later because we want to keep the assumptions of different sections distinct.

There is some $\lambda' \in \C$ such that $T(\lambda')^{-1} : L^2(\O) \to H^4 \cap H^2_0(\O)$ is bounded by Proposition \ref{invprop}. Moreover
\begin{equation}
T(\lambda') - T(\lambda) = (\lambda' - \lambda) P_2 + (\lambda'^2 - \lambda^2) P_0
\end{equation}
where $P_2$ is a second order partial differential operator and $P_0$ just a smooth function. Hence the difference maps $H^4 \cap H^2_0(\O) \to H^2(\O)$ and the latter embeds compactly into $L^2(\O)$ by Remark \ref{schattenRem}. Now
\begin{equation}
\lambda \mapsto \operatorname{Id} - T(\lambda)T(\lambda')^{-1} = (T(\lambda') - T(\lambda))T(\lambda')^{-1}
\end{equation}
is an analytic family of compact operators in $L^2(\O)$. By the analytic Fredholm theorem (see for example the supplementary note 3 for chapter VII concerning Theorem VII.1.9 in \cite{Kato}), we see that $T(\lambda)T(\lambda')^{-1}$ is meromorphic, and hence $T(\lambda)^{-1}$ is so too. The location of its poles follows from Proposition \ref{invprop}. The description of the neighborhood is due to the fact that $R$ may depend on $Q$. But we still know that if $\Gamma \neq \R_-$ is a ray starting from the origin, then $\Gamma$ will not have any poles on it after some finite distance. Hence the requirement for $r(t)t^{-1} \to 0$.

Let $\Gamma \subset \C$ be a ray not touching any of the poles and not pointing in the same direction as $\R_-$. Let $v$ be the direction vector of $\Gamma$ and take $Q$ to be 
\begin{equation}
\{z \mid (1-\epsilon) \arg (v_1 + i v_2) \leq \arg z \leq (1 + \epsilon) \arg (v_1 + i v_2) \}
\end{equation}
for a sufficiently small $\epsilon > 0$ so that $Q \cap \R_- = \emptyset$. Take $R$ from Proposition \ref{invprop}. Now the estimate holds on $Q \setminus B(0,R)$. Then $\Gamma \setminus (Q \setminus B(0,R))$ is bounded, and can be enclosed in a compact set $K$ not touching any of the poles. Our operator family is analytic in a neighborhood of $K$, hence it is bounded on $K$. This implies the estimate for the whole ray.
\end{proof}

\newpage
\section{Growth of order $p$}
\label{growthSect}
We want to give an estimate for $T(\lambda)^{-1}$ on large circles in this section. More precisely, we want to show that 
\begin{equation}
\abs{ (T(\lambda)^{-1}f, g) } \leq e^{C \abs{\lambda}^p}
\end{equation}
on circles $\abs{\lambda} = R_j$, $R_j \to \infty$, whenever $f$ and $g$ are such that the inner product is entire.

As we saw in the introduction, there are many ways of proving that, all of them related to each other. Hence we shall prove it in two ways. The first one works for analytic families $T(\lambda)$ but uses Nevanlinna characteristics. The second one uses mostly just basic functional analysis but only works for polynomials. On the other hand, the second approach gives a better Weyl law for the poles.

\subsection*{Proof using Nevanlinna characteristics}
We will use \emph{Nevanlinna characteristics}, namely the well-known $\Tinf$ and the more specific $\Tone$ from \cite{nevanlinna}. Both of them are defined in that same article.

\begin{definition}
Let $\lambda \mapsto W(\lambda)$ be meromorphic $\C \to \mathscr{L}(X,Y)$. Then we define
\begin{multline}
\Tinf(r,W) = \frac{1}{2\pi} \int_{-\pi}^\pi \ln^+\norm{W(r e^{i\theta})} d\theta \\
+ \int_0^r \frac{n_\infty(t,W) - n_\infty(0,W)}{t} dt + n_\infty(0,W) \ln r,
\end{multline}
where $n_\infty(t, W)$ counts the number of poles of $W$ in the closed disc $\abs{\lambda} \leq t$ with multiplicities.
\end{definition}
\begin{remark}
Its most remarkable property for us is that if $W$ is analytic, the maximum of $W$ over circles can be estimated by $\Tinf$ on slightly bigger circles.
\end{remark}

\medskip
From now on, we assume that $H$ is a Hilbert space and that $D \subset H$ is a dense normed subspace. We write $\mathscr{C}^p$ for the Schatten class of order $p$. Its norm is given by the $\ell^p$-norm of the sequence of \emph{singular-}, or \emph{characteristic values}. See for example \cite[XI.9]{DunfordSchwartz2}.
\begin{theorem}
\label{inverseEstimate}
Let $\lambda \mapsto T(\lambda)$ be analytic $\C \to \mathscr{L}(D,H)$. Assume that there exists some $\lambda' \in \C$ and $1 \leq p < \infty$ such that 
\begin{itemize}
\item $T(\lambda')$ is invertible and
\item $\lambda \mapsto T(\lambda') - T(\lambda)$ is analytic $\C \to \mathscr{C}^p(D,H)$.
\end{itemize}
Then $\lambda \mapsto T(\lambda)^{-1}$ is meromorphic $\C \to \mathscr{L}(H,D)$ and
\begin{equation}
\Tinf(r,T^{-1}) \leq C\big( 1 + \sup_{\abs{\lambda}=r} \ln(1+ \norm{T(\lambda)}) + \sup_{\abs{\lambda} = r} \norm{T(\lambda') - T(\lambda)}_{\mathscr{C}^p}^{\lceil p \rceil} \big),
\end{equation}
where $C$ depends only on $T$, $\lambda'$ and $p$. The number $\lceil p \rceil$ is the smallest integer at least $p$.
\end{theorem}
\begin{remark}
Actually $T(\lambda)^{-1}$ will be \emph{finite meromorphic}, which means that the coefficients of the negative powers in its Laurent expansions are operators of finite rank, but we won't need that information.
\end{remark}
\begin{proof}
We will consider the operator valued function $\lambda \mapsto F(\lambda)$ given by
\begin{equation}
F(\lambda) = \operatorname{Id} - T(\lambda)T(\lambda')^{-1} = (T(\lambda') - T(\lambda)) T(\lambda')^{-1} \in \mathscr{C}^p(H,H).
\end{equation}
The meromorphicity of $T(\lambda)^{-1}$ will follow from the fact that the family of operators $T(\lambda)T(\lambda')^{-1} = \operatorname{Id} - F(\lambda)$ is Fredholm and analytic. For the proof, see \cite[4.1.4]{gohbergLeiterer} or \cite{Kato}, like in Theorem \ref{merothm}.

Let us now prove the estimate. Let $m = \lceil p \rceil$ so that $m$ is the smallest integer such that $F^m \in \mathscr{C}^1(H,H)$. This exists since $\norm{F^m}_{\mathscr{C}^1} \leq \norm{F}^m_{\mathscr{C}^m}$ and since we have the inclusion $\mathscr{C}^{\lceil p \rceil} \hookrightarrow \mathscr{C}^p$. Then use Theorem 3.2 from \cite{nevanlinna} to get
\begin{multline}
\Tinf (r, T(\lambda')T(\lambda)^{-1}) = \Tinf\big(r, (\operatorname{Id} - F(\lambda))^{-1}\big) \\
\leq \Tone (r, \operatorname{Id}-F^m) + (m-1)\big( \Tinf (r, F) + \ln 2\big) - \ln \abs{c_{-\nu}}.
\end{multline}
Let's look how each term is defined in \cite{nevanlinna} and estimate them.
\begin{enumerate}[a)]
\item The constant $c_{-\nu}$ is the first nonzero coefficient of the Laurent series of $\det(\operatorname{Id} - F^m)$ at the origin. This series is well defined since $F^m \in \mathscr{C}^1(H,H)$ can be approximated by finite rank operators. Nonetheless $c_{-\nu}$ does not depend on $r$ and $0 < \abs{c_{-\nu}} < \infty$.
\item Since $\lambda \mapsto F(\lambda)$ is analytic, we have
\begin{multline}
\Tinf (r, F) = \frac{1}{2\pi} \int_{-\pi}^\pi \ln^+ \norm{F(r e^{i\theta})}_{H\to H} d\theta \\
\leq \frac{1}{2\pi} \int_{-\pi}^\pi \ln^+\big(1 + \norm{T(re^{i\theta})}_{D\to H} \norm{T(\lambda')^{-1}}_{H\to D}\big) d\theta \\
\leq C_{T, \lambda'} \sup_{\abs{\lambda}=r} \ln(1 + \norm{T(\lambda)}_{D\to H}).
\end{multline}
\item Again, $\operatorname{Id} - F^m$ is analytic, so according to definitions 2.4, 2.5, 2.9 and lemmas 2.2 and 2.8 in \cite{nevanlinna}, we get
\begin{multline}
\Tone (r, \operatorname{Id} - F^m) = \frac{1}{2\pi} \int_{-\pi}^\pi s\big(\operatorname{Id} - F(r e^{i\theta})^m\big) d\theta \\
= \frac{1}{2\pi} \int_{-\pi}^\pi \sum_{j=0}^\infty \ln^+ \sigma_j\big( \operatorname{Id} - F(re^{i\theta})^m\big) d\theta \leq \frac{1}{2\pi} \int_{-\pi}^\pi \norm{F(re^{i\theta})^m}_{\mathscr{C}^1(H,H)} d\theta \\
\leq C_p \int_{-\pi}^\pi \norm{F(re^{i\theta})}_{\mathscr{C}^p(H,H)}^m d\theta \leq C_{p, T, \lambda'} \sup_{\abs{\lambda} = r} \norm{T(\lambda') - T(\lambda)}_{\mathscr{C}^p(D,H)}^m.
\end{multline}
\end{enumerate}
The claim follows now by using Theorem 2.1 in \cite{nevanlinna}, which gives us
\begin{equation}
\Tinf \big(r, T(\lambda)^{-1}\big) \leq \Tinf \big(r, T(\lambda') T(\lambda)^{-1}\big) + \Tinf \big(r, T(\lambda')^{-1} \big),
\end{equation}
where the second term is just $\ln^+ \norm{T(\lambda')^{-1}}_{H\to D} = C_{T,\lambda'} < \infty$.
\end{proof}

\begin{corollary}
\label{inverseEstCor1}
Let $T$ satisfy all the assumptions in Theorem \ref{inverseEstimate}. Assume that there are $f, g \in H$, $\norm{f} = \norm{g} = 1$ such that $\lambda \mapsto \big( T(\lambda)^{-1} f, g \big)$ is entire. Then
\begin{multline}
\sup_{\abs{\lambda}=r} \ln \abs{\big( T(\lambda)^{-1}f,g\big)} \leq \\
 C \left( 1 + \sup_{\abs{\lambda}=2r} \ln( 1 + \norm{T(\lambda)}_{D\to H}) +  \sup_{\abs{\lambda}=2r} \norm{T(\lambda) - T(\lambda)}_{\mathscr{C}^p(D,H)}^{\lceil p \rceil} \right)
\end{multline}
for $r > 0$.
\end{corollary}
\begin{proof}
See for example Paragraph 8.9, 8.91 in \cite{Titch} or Theorem 2.2 in \cite{nevanlinna}. By them
\begin{equation}
 \sup_{\abs{\lambda}=r} \abs{\big( T(\lambda)^{-1}f,g\big)} = \operatorname{M}_\infty \big(r, \big( T(\lambda)^{-1}f,g\big) \big) \leq e^{ 3r \Tinf\left( 2r, ( T(\lambda)^{-1}f,g) \right)},
\end{equation}
and then
\begin{multline}
\Tinf \big(2r, \big( T(\lambda)^{-1}f,g\big) \big) = \frac{1}{2\pi} \int_{-\pi}^\pi \ln^+ \abs{\big( T(2r e^{i\theta})^{-1}f,g\big)} d\theta \\
\leq \frac{1}{2\pi} \int_{-\pi}^\pi \ln^+ \norm{ T(2r e^{i\theta})^{-1} } d\theta \leq \Tinf (2r, T(\lambda)^{-1}).
\end{multline}
\end{proof}

\begin{corollary}
\label{inverseEstCor2}
Let $T \in \{T_H, T_S\}$ be like in Section \ref{elliptParamSect}, $p > n/2$, and assume that $f,g \in L^2(\O)$, $\norm{f} = \norm{g} = 1$, are such that $\lambda \mapsto \big( T(\lambda)^{-1} f, g\big)$ is entire. Then
\begin{equation}
\abs{\big( T(\lambda)^{-1}f,g\big)} \leq \exp C\abs{\lambda}^{2\lceil p \rceil}
\end{equation}
for all $\abs{\lambda} \geq 1$.
\end{corollary}
\begin{proof}
This follows from Corollary \ref{inverseEstCor1}, the fact that the highest power of $\lambda$ in $T(\lambda)$ is two and the fact that $T(\lambda') - T(\lambda) : H^4\cap H^2_0(\Omega) \to H^2(\Omega)$. Embedding the latter into $L^2(\Omega)$ is Schatten $p > n/2$ by Remark \ref{schattenRem}.
\end{proof}

\bigskip
If we want to prove a Weyl law for the poles of $T_H(\lambda)^{-1}$ or $T_S(\lambda)^{-1}$, the previous results are not optimal. This may be due to the fact that the coefficient of $\lambda^2$ is just a smooth function, which is in a smaller Schatten class than $\mathscr{C}^p$ with $p>n/2$. The latter is needed for the coefficient of $\lambda$, a second order differential operator.

\subsection*{Proof using Carleman's resolvent inequality}
If $T(\lambda)^{-1}$ were a resolvent, we could use an estimate of the form
\begin{equation}
\norm{\varphi(\lambda) (K - \lambda)^{-1}} \leq \abs{\lambda} \exp {\frac{1}{2}\left(1 + \frac{\norm{K}_{\mathscr{C}^2}^2}{\abs{\lambda}^2}\right)}
\end{equation}
for Hilbert-Schmidt operators $K$ and a specific analytic function $\varphi$. The estimate is from Carleman and it has a generalization to Schatten class operators, see \cite{carleman} and \cite[XI.6.27, XI.9.25]{DunfordSchwartz2}. We have to linearize $T(\lambda)$ first. The constructions of this section work for any polynomial family of operators with some assumptions on the coefficients, but we shall do it only for $T=T_H$ and $T=T_S$.

We construct a $2 \times 2$ matrix operator $\A$ such that $T^{-1}(\lambda)$ is an entry in $(\A - \lambda)^{-1}$. We can then use Carleman's resolvent estimate to get an upper bound for $T^{-1}(\lambda)$. This is the same linearization trick that is used to reduce ordinary differential equations of order $n$ to a first order $n \times n$ matrix differential equation.

Throughout this section we will assume that $T \in \{T_H, T_S\}$, where
\begin{equation}
\begin{split}
T_H : \,&H^4\cap H^2_0(\O) \to L^2(\O)\\
& u \mapsto  \Delta q \Delta u - \lambda(\Delta q + q\Delta + \Delta) u + \lambda^2(1+q) u \\
T_S : \,&H^4\cap H^2_0(\O) \to L^2(\O)\\
& u \mapsto  \Delta q \Delta u + \Delta - \lambda(\Delta q + q\Delta + 1) u + \lambda^2q u
\end{split}
\end{equation}
are the fourth-order differential operators corresponding to the interior transmission eigenvalues of the Helmholtz and Schr\"odinger equations respectively. We will also assume that $q \in C^\infty(\overline\O)$ is real-valued and bounded away from zero and infinity. We write
\begin{equation}
T(\lambda) = A_0 + \lambda A_1 + \lambda^2 A_2.
\end{equation}
By writing $v = \lambda A_2 u$ we see that $T(\lambda) u = 0$ if and only if there is $v \in H^2_0(\O)$ such that
\begin{equation}
\begin{cases}
\lambda u = A_2^{-1} v,\\
\lambda v = -A_0 u - A_1 A_2^{-1} v.
\end{cases}
\end{equation}
Hence the ``non-linear eigenvalue problem'' $T(\lambda) u = 0$ can be reduced to a linear matrix-valued eigenvalue problem $\A \vec u = \lambda \vec u$.

\begin{definition}
Let $\A$ be the operator $H^4\cap H^2_0 (\O) \times H^2_0(\O) \to H^2_0(\O) \times L^2(\O)$ defined by
\begin{equation}
\A \bm u \\ v \sm = \bm 0 & A_2^{-1} \\ -A_0 & -A_1 A_2^{-1} \sm \bm u \\ v \sm,
\end{equation}
where $A_0$, $A_1$ and $A_2$ are the coefficients of $T(\lambda)$.
\end{definition}

\begin{lemma}
\label{TtoALemma}
The operator $T(\lambda)$ is invertible if and only if $\A - \lambda$ is, and
\begin{equation}
\norm{T^{-1}(\lambda)}_{L^2(\O) \to L^2(\O)} \leq \norm{ (\A - \lambda)^{-1} }_{H^2_0(\O) \times L^2(\O) \to H^2_0(\O) \times L^2(\O)}.
\end{equation}
\end{lemma}
\begin{proof}
By an elementary calculation we see that
\begin{equation}
(\A - \lambda)^{-1} = \bm -T(\lambda)^{-1}(A_1 + \lambda A_2) & -T(\lambda)^{-1} \\ A_2 - A_2 T(\lambda)^{-1} (\lambda A_1 + \lambda^2 A_2) & -\lambda A_2 T(\lambda)^{-1} \sm 
\end{equation}
Hence if $u \in L^2(\O)$, then
\begin{multline}
\norm{ T(\lambda)^{-1} u }_{L^2(\O)} \leq \sqrt{\norm{-T(\lambda)^{-1} u}_{L^2(\O)}^2 + \norm{-\lambda A_2 T(\lambda)^{-1} u}_{L^2(\O)}^2} \\
= \norm{ (\A - \lambda)^{-1} \bm 0 \\ u \sm }_{L2(\O) \times L^2(\O)} \leq \norm{(\A - \lambda)^{-1}} \norm{u}_{L^2(\O)},
\end{multline}
where the last operator norm is taken inside $H^2_0(\O) \times L^2(\O)$. We could as well get the estimate $\norm{T(\lambda)^{-1} u}_{H^4\cap H^2_0} \leq \norm{ (\A - \lambda)^{-1}}_{H^2_0 \times L^2 \to H^4\cap H^2_0 \times H^2_0} \norm{u}_{L^2}$.
\end{proof}

The idea of the next lemma is to allow us to use Carleman's resolvent inequality, which holds for resolvents of Schatten operators. Hence we have to write the resolvent of $\A$ as a resolvent of a compact operator, which will be another resolvent in this case.
\begin{lemma}
\label{compRes}
There is $\lambda' \in \C$ such that $(\A - \lambda')^{-1}$ is in $\mathscr{C}^p(H^2_0(\O) \times L^2(\O))$ for $p > n/2$, and for $\lambda$ in the resolvent set of $\A$, we have
\begin{equation}
(\A - \lambda)^{-1} = (\A - \lambda')^{-1} \big( \operatorname{Id} - (\lambda - \lambda')(\A - \lambda')^{-1}\big)^{-1}.
\end{equation}
\end{lemma}
\begin{proof}
The existence of $\lambda'$ follows from the invertibility of $T(\lambda)$ for some $\lambda$ by Proposition \ref{invprop}. The resolvent is Schatten since $\iota : H^2_0(\O) \to L^2(\O)$ and $\iota' : H^4(\O) \to H^2(\O)$ are so according to Lemma \ref{schattenIncl} and the remark after it. The identity is an elementary calculation following from the resolvent identity. This can be seen by operating both sides from the right by $\Id - (\lambda - \lambda')(\A - \lambda')^{-1}$.
\end{proof}

By Lemma \ref{compRes} the operator $\A$ has compact resolvent, hence its spectrum consists of a sequence $\lambda_0, \lambda_1, \ldots$ of eigenvalues. We take into account their multiplicities and arrange them so that $\abs{\lambda_0} \leq \abs{\lambda_1} \leq \ldots \to \infty$. Note that $(\lambda - \lambda')(\A - \lambda')^{-1} \in \mathscr{C}^p$ for all $\lambda$ and its eigenvalues are
\begin{equation}
\frac{\lambda - \lambda'}{\lambda_0 - \lambda'}, \frac{\lambda - \lambda'}{\lambda_1 - \lambda'}, \ldots
\end{equation}
The operator $( \operatorname{Id} - (\lambda - \lambda')(\A - \lambda')^{-1} )^{-1}$ has poles. To get a norm estimate for $(\A - \lambda)^{-1}$ without the singularities on the right-hand side causing trouble, we have to multiply by an analytic function having zeros at the poles. This is more or less the idea behind the determinant $\varphi$ in Carleman's inequality. We use the standard Weierstrass product for constructing such a function. See for example \cite{rudinRC} Chapter 15, especially Theorem 15.9.

\begin{lemma}
\label{phiLemma}
Let $p > n/2$, denote $k-1 \leq p \leq k \in \N$ and define 
\begin{equation}
\varphi(\lambda) = \prod_{j=0}^\infty \left(1- \frac{\lambda - \lambda'}{\lambda_j - \lambda'}\right) \exp \left( \frac{\lambda - \lambda'}{\lambda_j - \lambda'} + \cdots + \frac{1}{k-1} \left( \frac{\lambda - \lambda'}{\lambda_j - \lambda'} \right)^{k-1} \right).
\end{equation}
Then $\varphi$ is a well-defined entire function vanishing at all $\lambda_j$ counting multiplicities. Moreover it is of order $p$:
\begin{equation}
\abs{\varphi(\lambda)} \leq C \exp\left(C \norm{(\A - \lambda')^{-1}}_{\mathscr{C}^p}^p \abs{\lambda}^p\right),
\end{equation}
where the Schatten norm is taken inside $H^2_0(\O) \times L^2(\O)$.
\end{lemma}
\begin{proof}
Note that $(\lambda - \lambda')(\A - \lambda')^{-1} \in \mathscr{C}^p$, so the sequence of its eigenvalues, counting multiplicities, is in $\ell^p$. Hence by Theorem 15.9 of \cite{rudinRC} the function $\varphi$ is well-defined, entire and vanishes on the eigenvalues. The estimate comes from
\begin{equation}
\abs{(1-z)e^{z + \frac{z^2}{2} + \cdots + \frac{z^{k-1}}{k-1}} } \leq \begin{cases}e^{k\abs{z}^{k-1}}, &\text{when } \abs{z} > 1\\ e^{\abs{z}^k}, &\text{when } \abs{z}\leq 1 \end{cases}
\end{equation}
and the fact that $\norm{(\lambda_j - \lambda')^{-1}}_{\ell^p} \leq \norm{ (\A - \lambda')^{-1} }_{\mathscr{C}^p}$.
\end{proof}

\begin{proposition}[Carleman's inequality]
\label{carlemanProp}
There is a constant $C = C(p,\lambda', \A)$ such that
\begin{equation}
\norm{ \varphi(\lambda) \big( \operatorname{Id} - (\lambda - \lambda')(\A - \lambda')^{-1} \big)^{-1} }_{\mathscr{L}(H^2_0(\O) \times L^2(\O))} \leq C e^{C\abs{\lambda}^p},
\end{equation}
for all $\lambda \in \C$.
\end{proposition}
\begin{proof}
Use Corollary XI.9.25 on page 1112 of \cite{DunfordSchwartz2}. More explicitely, let their $T := - (\lambda-\lambda')(\A - \lambda')^{-1}$, which is Schatten by Lemma \ref{compRes}. Their $\operatorname{det}_k(\operatorname{Id} + T)$ is our $\varphi(\lambda)$. See Definition XI.9.21 on page 1106.
\end{proof}

\begin{theorem}
\label{TonCircles}
Let $\O \subset \R^n$ be a bounded domain, $q \in C^\infty(\overline{\O})$ positive, bounded away from zero and infinity and let $p > n/2$. Then there are real numbers $0 < r_0 < r_1 < \ldots \to \infty$ such that $T(\lambda)$ is invertible on $\abs{\lambda} = r_0, r_1, \ldots$ and
\begin{equation}
\norm{T^{-1}(\lambda)}_{L^2(\O) \to L^2(\O)} \leq e^{C_\varepsilon \abs{\lambda}^{p+\varepsilon}}, \quad \abs{\lambda} \in \{r_0, r_1, \ldots \}
\end{equation}
for any $\varepsilon > 0$.
\end{theorem}
\begin{proof}
By Lemma \ref{phiLemma} the function $\varphi$ is entire and of order $p$. Check out the results 8.71 and 8.711 on page 273 of Titchmarsh \cite{Titch}. They tell that $\varphi(\lambda)$ decreases at around the speed of $e^{-\abs{\lambda}^p}$. More precicely, they imply that for any $\varepsilon > 0$ there are $0 < r_0 < r_1 < \ldots \to \infty$ such that the circles $S(0,r_j)$ do not touch the zeros of $\varphi$, which are the poles $\lambda_0, \lambda_1, \ldots$ of $(\lambda - \A)^{-1}$, and that
\begin{equation}
\abs{\varphi(\lambda)} > e^{- r_j^{p+\varepsilon}} \text{ on } \abs{\lambda} = r_j.
\end{equation}
Hence by Lemma \ref{TtoALemma}, Lemma \ref{compRes} and Proposition \ref{carlemanProp} we have
\begin{multline}
\norm{T^{-1}(\lambda)} \leq \norm{(\A - \lambda)^{-1}} = \norm{(\A - \lambda')^{-1} \left( \operatorname{Id} - (\lambda - \lambda')(\A - \lambda')^{-1} \right)^{-1} } \\
\leq \norm{ (\A-\lambda')^{-1}} \abs{\varphi(\lambda)^{-1}} \norm{\varphi(\lambda) \left( \operatorname{Id} - (\lambda -\lambda')(\A-\lambda')^{-1} \right)^{-1}} \\
\leq \norm{(\A - \lambda')^{-1}} e^{r_j^{p+ \varepsilon}} C e^{Cr_j^p} \leq e^{C r_j^{p+\varepsilon}},
\end{multline}
when $\abs{\lambda} = r_j$. The constant $C$ at the end depends on $\lambda', \A, r_0, p$ and $\varepsilon$.
\end{proof}

\begin{corollary}
Let $\O \subset \R^n$ be a bounded domain, $q \in C^\infty(\overline{\O})$ positive, bounded away from zero and infinity and let $p > n/2$. Assume that $f,g \in L^2(\O)$ with $\norm{f} = \norm{g} = 1$. Then there are real numbers $0 < r_0 < r_1 < \ldots \to \infty$ such that
\begin{equation}
\abs{\big( T(\lambda)^{-1}f,g\big)} \leq \exp C_\varepsilon \abs{\lambda}^{p + \varepsilon}
\end{equation}
on $\abs{\lambda} = r_0, r_1, \ldots$
\end{corollary}
\begin{proof}
This follows directly from Theorem \ref{TonCircles}.
\end{proof}

\medskip
We can prove an upper bound for the counting function of the transmission eigenvalues after knowing that $\lambda_j^{-1} \in \ell^p$. The result is probably not the best one. See for example Lakshtanov and Vainberg \cite{LakshtanovVainberg} and other recent papers from the same authors, where they prove a Weyl law with exponent $n/2$ in non-isotropic cases. Using the theorems of \cite{Agranovich1992} or \cite{BoimatovKostyuchenko} it seems that our parameter-elliptic operator with weight $2$ would get a Weyl law with exponent $n/2$. Asymptotics of the counting function are not the topic of this paper, so we contend with just giving the following result which has a very simple proof.

\begin{theorem}
\label{WeylLawThm}
Let $\O \subset \R^n$ be a bounded domain, $q \in C^\infty(\overline{\O})$ positive, bounded away from zero and infinity and let $\varepsilon > 0$. If $\lambda_0, \lambda_1, \ldots$ are the poles of $T(\lambda)^{-1}$ counting multiplicities and $T$ is invertible at $\lambda' \in \C$, then
\begin{equation}
\# \{ j \in \N \mid \abs{\lambda_j - \lambda'} < t \} \leq C t ^{n/2 + \epsilon}
\end{equation}
with $C = \norm{(\A - \lambda')^{-1}}_{\mathscr{C}^{n/2 + \varepsilon}} < \infty$.
\end{theorem}
\begin{proof}
By Lemma \ref{TtoALemma} we know that $T(\lambda)^{-1}$ is not bounded exactly when $\lambda$ is an eigenvalue of $\A$, hence when $(\lambda - \lambda')^{-1}$ is an eigenvalue of $(\A - \lambda')^{-1}$. The latter is Schatten $n/2 + \varepsilon$ in $H^2_0(\O) \times L^2(\O)$ by Lemma \ref{compRes}. Hence its eigenvalues $(\lambda_j - \lambda')^{-1}$ and singular values $\sigma_j$ satisfy
\begin{equation}
\norm{ (\lambda_j - \lambda')^{-1} }_{\ell^{n/2 + \varepsilon}} \leq \norm{ \sigma_j }_{\ell^{n/2 + \varepsilon}} = \norm{ (\A - \lambda')^{-1} }_{\mathscr{C}^{n/2 + \varepsilon}} < \infty.
\end{equation}
Write $f(j) = \abs{(\lambda_j - \lambda')^{-1}}$. We want to estimate $\# \{ j \mid f(j) > t^{-1} \}$. By Chebyshev's inequality \cite{Chebyshev}, we get
\begin{multline}
\# \{ j \mid f(j) > t^{-1} \} \\
\leq \frac{1}{t^{-{(n/2 + \varepsilon)}}} \sum_{j=0}^\infty \abs{f(j)}^{n/2 + \varepsilon} = \norm{(\A - \lambda')^{-1}}_{\mathscr{C}^{n/2 + \varepsilon}} t^{n/2 + \varepsilon} .
\end{multline}
\end{proof}

\newpage
\section{The cases of Schr\"odinger and Helmholtz equations}
\label{lastSection}
We will now show that all the three assumptions of Section \ref{meroSect} are true for both of $T_H$ and $T_S$ given by
\begin{equation}
\begin{split}
T_H : \,&H^4\cap H^2_0(\O) \to L^2(\O)\\
& u \mapsto  \Delta q \Delta u - \lambda(\Delta q + q\Delta + \Delta) u + \lambda^2(1+q) u \\
T_S : \,&H^4\cap H^2_0(\O) \to L^2(\O)\\
& u \mapsto  \Delta q \Delta u + \Delta - \lambda(\Delta q + q\Delta + 1) u + \lambda^2q u
\end{split}
\end{equation}
We will use Theorem \ref{merothm} and Theorem \ref{TonCircles}.

\begin{theorem}
Let $\O \subset \R^n$ be a bounded domain, $q \in C^\infty(\overline{\O})$ positive, bounded away from zero and infinity and let $T = T_H$ or $T = T_S$. Then the set of generalized eigenstates corresponding to all singular values of $\lambda \mapsto T(\lambda)$ is complete in $L^2(\O)$.
\end{theorem}
\begin{proof}
Assumption 1 is true because $\lambda \mapsto T(\lambda)$ is a polynomial. For the two other ones, let $\beta = p + \varepsilon$ for some $p > n/2$ and $\varepsilon > 0$. Take $\lambda' \in \R_+$ such that $T(\lambda')$ is invertible. This exists by Proposition \ref{invprop}. Now draw a ray $\Gamma_j$ from $\lambda'$ to  the pole $\lambda_j$ of $T^{-1}(\lambda)$, and let $\Gamma_{-1} = \lambda' + \R_-$. These form a countable set by Theorem \ref{merothm}, and they are the only rays starting at $\lambda'$ on which
\begin{equation}
\norm{T^{-1}(\lambda)} \leq C \abs{\lambda}^{-2}
\end{equation}
will not necessarily hold. Hence it is possible to satisfy Assumption 2 no matter how small $\pi / \beta > 0$ is. Assumption 3 has been shown to be true in the corollaries after Theorem \ref{inverseEstimate} and Theorem \ref{TonCircles}. The claim follows from Theorem \ref{denseThm}.
\end{proof}

\begin{remark}
This method actually shows too that there is an infinite number of transmission eigenvalues, and that they form a discrete set. Discreteness follows from the fact that $\lambda \mapsto T(\lambda)^{-1}$ is meromorphic. Existence follows because the generalized eigenstates are dense in $L^2(\O)$. Given a single transmission eigenvalue $\lambda_0$, the space spanned by the corresponding generalized eigenstates $\operatorname{span} \sp(\lambda_0)$ is finite dimensional. To see this, check \cite[Chp I. 4]{KeldyshENG}. Substitute
\begin{equation}
A(\lambda) = \big( T(\lambda') - T(\lambda) \big) T(\lambda')^{-1},
\end{equation}
for his operator $A(\lambda)$. The point $\lambda'$ is any point of invertibility.

We showed that $\cup_{j=0}^\infty \operatorname{span} \sp(\lambda_j)$ is dense in $L^2(\O)$, and each span is finite dimensional. Hence the set $\{ \lambda_j \mid j \in \N \}$ is infinite.
\end{remark}

\medskip
We end this paper by mentioning some generalized interior transmission eigenvalue problems for which this method seems to work. The first one was mentioned already in the introduction. Namely, finding nontrivial $v,w \in H^4$ such that
\begin{equation}
\label{newITE}
\begin{aligned}
\big(\Delta + k^2(1+m)\big) w &= 0 \quad \text{in } \Omega\\
(\Delta + k^2) v &= 0 \quad \text{in } \Omega\\
\partial_\nu^{m_1}(w - v) &= 0 \quad\text{on } \partial\Omega \\
\partial_\nu^{m_2}(w - v) &= 0 \quad\text{on } \partial \Omega
\end{aligned}
\end{equation}
where $m_1 \neq m_2$ are any numbers from $\{0, 1, 2, 3\}$.

Again, write $\lambda = -k^2$ and $q = \tfrac{1}{m}$. Then problem \eqref{newITE} has a nontrivial solution if and only if the problem
\begin{equation}
\begin{aligned}
T(\lambda) u&= \Delta q \Delta u - \lambda(\Delta q + q\Delta + \Delta) u + \lambda^2(1+q) u = 0 ,\\
u &\in H^4_*(\Omega) .
\end{aligned}
\end{equation}
has a nontrivial solution. We will in fact have $u = w - v$ and $ w = q(\Delta - \lambda) u$, $v = q\left( \Delta - \lambda(1 + 1/q) \right)u$. But here we have
\begin{equation}
H^4_*(\Omega) = \{ u \in H^4(\Omega) \mid \partial_\nu^{m_1} u|_{\partial\Omega} = 0, \,\partial_\nu^{m_2} u|_{\partial \Omega} = 0 \}.
\end{equation}

We have to check that the three assumptions of Section \ref{meroSect} hold. The space $H^4_*(\Omega)$ is a Hilbert space dense in $L^2(\O)$ and $T(\lambda)$ is clearly analytic. The ellipticity conditions of Agranovich and Vishik are checked as in Proposition \ref{cond1prop} and Proposition \ref{cond2prop}. The only difference is caused by the boundary terms, and will happen when solving \eqref{helmholtzODE}. The method of characteristic polynomial still gives the exact same general solution. When choosing a particular solution when $\lambda = 0$ we will arrive at the determinant
\begin{equation}
\left\vert \begin{matrix} (-\abs{\xi'})^{m_1} & (-\abs{\xi'})^{m_1-1}m_1\\ (-\abs{\xi'})^{m_2} & (-\abs{\xi'})^{m_2-1}m_2 \end{matrix} \right\rvert \neq 0
\end{equation}
When $\lambda \neq 0$, the determinant is
\begin{equation}
\left\vert \begin{matrix} r_2^{m_1} & r_4^{m_1} \\ r_2^{m_2} & r_4^{m_2} \end{matrix} \right\rvert = r_2^{m_1}r_4^{m_2} - r_2^{m_2}r_4^{m_1} = r_2^{m_1}r_4^{m_1}\big(r_4^{m_2-m_1} - r_2^{m_2-m_1}\big) \neq 0
\end{equation} 
because $m_1 \neq m_2$ and $r_2 \neq r_4$. Now Agranovich and Vishik give us
\begin{equation}
\norm{ T(\lambda)^{-1} f }_{L^2(\O)} \leq C \abs{\lambda}^{-2} \norm{f}_{L^2(\O)}
\end{equation}
for $\lambda$ big enough in cones not touching $\R_-$. This will imply Assumpion 2 from Section \ref{meroSect}. To prove Assumption 3, just follow the proof of Corollary \ref{inverseEstCor2}.

\medskip
Other possible generalizations include having a non-isotropic metric, for example by switching $\Delta$ to $\nabla \cdot K \nabla$ with positive symmetric $K$. This kind of equation with Robin boundary conditions comes from optical tomography. With this change, only the conditions of Agranovich and Vishik need to be checked, which may still prove to be a tedious calculation. Continuing on this line, we can change $\Delta$ into any other elliptic operator, but again, the calculations may become tedious.

Another generalization is to change the homogeneous boundary conditions $u = 0$, $\partial_\nu u = 0$ to inhomogeneous ones $u = g_1$, $\partial_\nu u = g_2$. Here Agranovich and Vishik's result works immediately since it does not care about whether the equations are homogeneous or not. Instead, problems arise in the other parts of the proof because our operator $T(\lambda)$ will then be defined on an affine space which is not a vector space. Fixing this would require some changes which are out of the scope of this paper.

\bibliographystyle{plain}
\bibliography{DensBib}

\begin{thebibliography}{10}

\bibitem{AgmonLecturesOnElPDE}
Shmuel Agmon.
\newblock {\em Lectures on elliptic boundary value problems}.
\newblock Prepared for publication by B. Frank Jones, Jr. with the assistance
  of George W. Batten, Jr. Van Nostrand Mathematical Studies, No. 2. D. Van
  Nostrand Co., Inc., Princeton, N.J.-Toronto-London, 1965.

\bibitem{Agranovich1976}
Mikhail~Semenovich Agranovich.
\newblock Summability of series in root vectors of non-self-adjoint elliptic
  operators.
\newblock {\em Functional Analysis and Its Applications}, 10(3):165--174, 1976.

\bibitem{Agranovich1990}
Mikhail~Semenovich Agranovich.
\newblock Non-self-adjoint problems with a parameter that are elliptic in the
  sense of {A}gmon---{D}ouglis---{N}irenberg.
\newblock {\em Functional Analysis and Its Applications}, 24(1):50--53, 1990.

\bibitem{Agranovich1992}
Mikhail~Semenovich Agranovich.
\newblock On modules of eigenvalues for non-self-adjoint
  {A}gmon---{D}ouglis---{N}irenberg elliptic boundary problems with a
  parameter.
\newblock {\em Functional Analysis and Its Applications}, 26(2):116--119, 1992.

\bibitem{AgranovichVishik}
Mikhail~Semenovich Agranovich and Mark~Iosifovich Vishik.
\newblock Elliptic problems with a parameter and parabolic problems of general
  type.
\newblock {\em Uspehi Mat. Nauk}, 19(3 (117)):53--161, 1964.

\bibitem{BoimatovKostyuchenko}
K.Kh. Boimatov and A.G. Kostyuchenko.
\newblock Spectral asymptotics of polynomial pencils of differential operators
  in bounded domains.
\newblock {\em Functional Analysis and Its Applications}, 25(1):5--16, 1991.

\bibitem{CCbook}
Fioralba Cakoni and David Colton.
\newblock {\em Qualitative methods in inverse scattering theory}.
\newblock Interaction of Mechanics and Mathematics. Springer-Verlag, Berlin,
  2006.
\newblock An introduction.

\bibitem{CakoniGintidesHaddar}
Fioralba Cakoni, Drossos Gintides, and Houssem Haddar.
\newblock The existence of an infinite discrete set of transmission
  eigenvalues.
\newblock {\em SIAM J. Math. Anal.}, 42(1):237--255, 2010.

\bibitem{carleman}
Torsten Carleman.
\newblock Zur theorie der linearen integralgleichungen.
\newblock {\em Mathematische Zeitschrift}, 9:196--217, 1921.

\bibitem{CPS}
David Colton, Lassi P{\"a}iv{\"a}rinta, and John Sylvester.
\newblock The interior transmission problem.
\newblock {\em Inverse Probl. Imaging}, 1(1):13--28, 2007.

\bibitem{DunfordSchwartz2}
Nelson Dunford and Jacob~Theodore Schwartz.
\newblock {\em Linear operators. {P}art {II}}.
\newblock Wiley Classics Library. John Wiley \& Sons Inc., New York, 1988.
\newblock Spectral theory. Selfadjoint operators in Hilbert space, With the
  assistance of William G. Bade and Robert G. Bartle, Reprint of the 1963
  original, A Wiley-Interscience Publication.

\bibitem{gohbergLeiterer}
Israel Gohberg and J{\"u}rgen Leiterer.
\newblock {\em Holomorphic operator functions of one variable and
  applications}, volume 192 of {\em Operator Theory: Advances and
  Applications}.
\newblock Birkh\"auser Verlag, Basel, 2009.
\newblock Methods from complex analysis in several variables.

\bibitem{HitrikKrupchykOlaPaivarintaTEconstant2010}
Michael Hitrik, Katsiaryna Krupchyk, Petri Ola, and Lassi P{\"a}iv{\"a}rinta.
\newblock Transmission eigenvalues for operators with constant coefficients.
\newblock {\em SIAM J. Math. Anal.}, 42(6):2965--2986, 2010.

\bibitem{HitrikKrupchykOlaPaivarintaTElocations2011}
Michael Hitrik, Katsiaryna Krupchyk, Petri Ola, and Lassi P{\"a}iv{\"a}rinta.
\newblock The interior transmission problem and bounds on transmission
  eigenvalues.
\newblock {\em Math. Res. Lett.}, 18(2):279--293, 2011.

\bibitem{HitrikKrupchykOlaPaivarintaEllipticOps2011}
Michael Hitrik, Katsiaryna Krupchyk, Petri Ola, and Lassi P{\"a}iv{\"a}rinta.
\newblock Transmission eigenvalues for elliptic operators.
\newblock {\em SIAM J. Math. Anal.}, 43(6):2630--2639, 2011.

\bibitem{Kato}
Tosio Kato.
\newblock {\em Perturbation theory for linear operators}.
\newblock Springer-Verlag, Berlin, second edition, 1976.
\newblock Grundlehren der Mathematischen Wissenschaften, Band 132.

\bibitem{KeldyshOrig}
Mstislav~Vsevolodovich Keldysh.
\newblock On the characteristic values and characteristic functions of certain
  classes of non-self-adjoint equations.
\newblock {\em Doklady Akad. Nauk SSSR (N.S.)}, 77:11--14, 1951.

\bibitem{KeldyshENG}
Mstislav~Vsevolodovich Keldysh.
\newblock The completeness of eigenfunctions of certain classes of
  nonselfadjoint linear operators.
\newblock {\em Uspehi Mat. Nauk}, 26(4(160)):15--41, 1971.

\bibitem{KatyaSHAPLOP}
Katsiaryna Krupchyk and Jukka Tuomela.
\newblock The {S}hapiro-{L}opatinskij condition for elliptic boundary value
  problems.
\newblock {\em LMS J. Comput. Math.}, 9:287--329, 2006.

\bibitem{LaiRobertNL}
PhamThe Lai and Didier Robert.
\newblock Sur un probleme aux valeurs propres non lineaire.
\newblock {\em Israel Journal of Mathematics}, 36(2):169--186, 1980.

\bibitem{LakshtanovVainbergBounds2012}
Evgeny Lakshtanov and Boris Vainberg.
\newblock Bounds on positive interior transmission eigenvalues.
\newblock {\em Inverse Problems}, 28(10):105005, 13, 2012.

\bibitem{LakshtanovVainbergEllipticity2012}
Evgeny Lakshtanov and Boris Vainberg.
\newblock Ellipticity in the interior transmission problem in anisotropic
  media.
\newblock {\em SIAM J. Math. Anal.}, 44(2):1165--1174, 2012.

\bibitem{LakshtanovVainberg}
Evgeny Lakshtanov and Boris Vainberg.
\newblock Remarks on interior transmission eigenvalues, {W}eyl formula and
  branching billiards.
\newblock {\em J. Phys. A}, 45(12):125202, 10, 2012.

\bibitem{Linden}
Hansj{\"o}rg Linden.
\newblock Linearization, completeness, and spectral asymptotics for certain
  rational and meromorphic operator functions.
\newblock {\em J. Operator Theory}, 39(2):219--247, 1998.

\bibitem{Lopatinsky}
Yaroslav~Borisovich Lopatinsky.
\newblock On a method of reducing boundary problems for a system of
  differential equations of elliptic type to regular integral equations.
\newblock {\em Dopovidi Akad. Nauk Ukrain. RSR}, 1952:381--388, 1952.

\bibitem{Markus}
Aleksandr~Semenovich Markus.
\newblock {\em Introduction to the spectral theory of polynomial operator
  pencils}, volume~71 of {\em Translations of Mathematical Monographs}.
\newblock American Mathematical Society, Providence, RI, 1988.
\newblock Translated from the Russian by H. H. McFaden, Translation edited by
  Ben Silver, With an appendix by M. V. Keldysh.

\bibitem{nevanlinna}
Olavi Nevanlinna.
\newblock Growth of operator valued meromorphic functions.
\newblock {\em Ann. Acad. Sci. Fenn. Math.}, 25(1):3--30, 2000.

\bibitem{PS}
Lassi P{\"a}iv{\"a}rinta and John Sylvester.
\newblock Transmission eigenvalues.
\newblock {\em SIAM J. Math. Anal.}, 40(2):738--753, 2008.

\bibitem{robbiano}
Luc Robbiano.
\newblock {Spectral analysis on interior transmission eigenvalues}.
\newblock {\em ArXiv e-prints}, February 2013.

\bibitem{LaiRobertENG}
Didier Robert.
\newblock Non linear eigenvalue problems.
\newblock {\em ArXiv Mathematical Physics e-prints}, December 2004.

\bibitem{rudinRC}
Walter Rudin.
\newblock {\em Real and complex analysis}.
\newblock McGraw-Hill Book Co., New York, third edition, 1987.

\bibitem{Shapiro}
Zorya~Yakovlevna Shapiro.
\newblock On general boundary problems for equations of elliptic type.
\newblock {\em Izv. Akad. Nauk SSSR Ser. Mat.}, 17(6):539--562, 1953.

\bibitem{Chebyshev}
Pafnouti~Lvovitch Tchebychev.
\newblock Des valeurs moyennes.
\newblock {\em Journal de math{\'e}matiques pures et appliqu{\'e}es},
  12(2):177--184, 1867.

\bibitem{Titch}
Edward~Charles Titchmarsh.
\newblock {\em The theory of functions}.
\newblock Oxford University Press, 2nd edition, 1952 (corrected) edition, 1939.

\end{thebibliography}

\end{document}